\documentclass[dvips,a4paper]{article}
\newlength{\marge}
\paperheight
\paperwidth
\usepackage{lmodern}
\usepackage{textgreek}
\usepackage[utf8]{inputenc}

\usepackage[breaklinks=true,colorlinks=true,citecolor=red,linkcolor=blue,urlcolor=red,unicode]{hyperref}

\title{On the (Fourier analytic) Sidon constant of \{0,1,2,3\}}
\author{Stefan Neuwirth}
\date{26 August 2013}

\usepackage[english]{babel}
\usepackage{textcomp}
\usepackage{amssymb,amsmath}
\usepackage{amsthm}
\theoremstyle{plain}
\newtheorem{thm}{Theorem}[section]
\newtheorem{prp}[thm]{Proposition}

\theoremstyle{definition}

\theoremstyle{remark}
\newtheorem{rem}[thm]{Remark}

\def\C{\mathbb C}
\def\Cont{\mathrm C}

\def\T{\mathbb T}
\def\Ell{\mathrm{L}}

\def\U{\mathbb U}
\def\e{\mkern1mu\mathrm e\mkern1mu}
\def\ei#1{\mkern1mu\mathrm e\mkern2mu^{\mathrm i#1}} 
\def\emi#1{\mkern1mu\mathrm e\mkern2mu^{-\mathrm i#1}} 
\def\iu{\mkern1mu\mathrm i\mkern1mu}
\def\Id{\mathrm{Id}}
\def\Z{\mathbb Z}
\providecommand{\abs}[1]{\lvert#1\rvert}
\DeclareMathOperator{\card}{\#}

\begin{document}
\maketitle

\begin{abstract}
  \noindent We study an elementary extremal problem on trigonometric polynomials
  of degree $3$. We discover a distinguished torus of extremal
  functions.%% $f_u(t)$ given by
\end{abstract}

\section{Introduction}

\noindent Let $\Lambda=\{\lambda_0,\lambda_1,\dots,\lambda_{n-1}\}$ be a set of $n$
frequencies and let $\T=\{z\in\C:|z|=1\}$. We study the following extremal problem: 

% \vskip\abovedisplayskip\setbox0=\hbox{$(\dagger)$\quad}\dimen0=\wd0 
% \noindent\box0
% $\vcenter{
% \advance\hsize by -\dimen0
% \noindent
% Given $n$ positive intensities
% $\varrho_0, \varrho_1, \dots, \varrho_{n-1}$, to find $n$ phases
% $\vartheta_0,\allowbreak \vartheta_1,\allowbreak \dots,\allowbreak \vartheta_{n-1}$ such that the maximum
% $\max_t|\sum\varrho_j\ei{\vartheta_j}\ei{\lambda_jt}|$ is minimal.}$
% \vskip\belowdisplayskip

% \noindent This should help us to study the
% following extremal problem:

\vskip\abovedisplayskip\setbox0=\hbox{$(\ddagger)$\quad}\dimen0=\wd0 
\noindent\box0
$\vcenter{
\advance\hsize by -\dimen0
\noindent
To find $n$ complex coefficients $c_0,c_1,\dots,c_{n-1}$ with given
moduli sum $\abs{c_0}+\abs{c_1}+\dots+\abs{c_{n-1}}=1$
such that the maximum
$\max_{z\in\T}|\sum c_jz^{\lambda_j}|$ is minimal.}$
\vskip\belowdisplayskip 

\noindent Note that this maximum's inverse is the \emph{Sidon
constant} $S(\Lambda)$. D.~J.~Newman (see \cite[Chapter
3]{sh51}) obtained the upper bound $S(\{0,1,\dots,N\})\le\sqrt N$ that is slightly better than the straightforward upper bound~$\sqrt{N+1}$: by Parseval's theorem for the
$\Ell^2$ space on the set~$\U_N$ of $N$th roots of unity, putting \[c_0+c_1z+\dots+c_{N-1}z^{N-1}+c_Nz^N=f(z)\text,\] we have
\begin{align}
\max_{z\in\T}|f(z)|^2
&=\max_{z\in\T}\max_{\omega\in\U_N}|f(z\omega)|^2\notag\\
&\ge\max_{z\in\T}\frac1N\sum_{\omega\in\U_N}|f(z\omega)|^2\notag\\
&=\max_{z\in\T}|c_0+c_Nz^N|^2+|c_1|^2+\dots+|c_{N-1}|^2\notag\\
&=\bigl(|c_0|+|c_N|\bigr)^2+|c_1|^2+|c_2|^2+\dots+|c_{N-1}|^2\label{newm}\\
&\ge\bigl(|c_0|+|c_1|+\dots+|c_N|\bigr)^2/N,\notag
\end{align}
and H.~S.~Shapiro showed (ibid.)\ that equality can
hold exactly if $N\in\{1,2,4\}$. % If $N=2$, equality holds
% exactly for multiples and translates of $f(z)=1+2\iu z+z^2$. 
If $N=3$, we shall show in the final section that the functions 
\vskip\abovedisplayskip
\centerline{$\displaystyle\frac{\iu2\sqrt2\cos \tau-1-3\sin \tau}{15} + \frac{3+\sin
\tau}{10}z + \frac{3-\sin \tau}{10} z^2 +
\frac{\iu2\sqrt2\cos \tau-1+3\sin \tau}{15} z^3$}\vskip\belowdisplayskip
\noindent have their modulus bounded by 
$3/5$ for each $\tau$, so that $5/3\le S(\{0,1,2,3\})<\sqrt{3}$.
\vskip\belowdisplayskip

A motivation for this problem is that we wish to know whether
the real and complex unconditionality constants are distinct for basic
sequences of characters $z^n$, but this remains undecided.

\section{Hervé Queffélec's proof}
When I showed Hervé Queffélec a proof that $S(\Lambda)\le\sqrt{n-1}$ for all sets $\Lambda$ with $n$~elements, he showed me how to adapt D.\ J.~Newman's argument to this more general case. 

Let $\Lambda$~be a set of $n$~frequencies. We may suppose that $\min\Lambda=0=\lambda_0$; let $N=\max\Lambda=\lambda_{n-1}$. Then
\begin{align*}
\max_{z\in\T}|f(z)|^2
&=\max_{z\in\T}\max_{\omega\in\U_N}|f(z\omega)|^2\notag\\
&\ge\max_{z\in\T}\frac1N\sum_{\omega\in\U_N}|f(z\omega)|^2\notag\\
&=\max_{z\in\T}|c_0+c_{n-1}z^{\lambda_{n-1}}|^2+|c_1|^2+\dots+|c_{n-2}|^2\notag\\
&=\bigl(|c_0|+|c_{n-1}|\bigr)^2+|c_1|^2+|c_2|^2+\dots+|c_{n-2}|^2\\
&\ge\bigl(|c_0|+|c_1|+\dots+|c_{n-1}|\bigr)^2/(n-1).\notag
\end{align*}
Let us now try to understand what is behind D.\ J.~Newman's argument.

\section{\texorpdfstring{Interpolating linear functionals on the space $\Cont_\Lambda(\T)$}{Interpolating linear functionals on the space C\_\textLambda(T)}}

\noindent If $L$ is a subspace of the space~$\Cont(T)$ of complex
continuous functions on a compact space~$T$ with $n$~dimensions, then
every functional~$l$ on~$L$ extends isometrically to a functional
on~$\Cont(T)$ by the Hahn-Banach theorem, that is, to a Radon measure
by the Riesz representation theorem. But the unit ball of the space of
measures is the weak*-closed convex hull of Dirac masses. By
Carathéodory's theorem for the space~$L$ that has $2n$~real
dimensions, $l$~extends isometrically to a linear combination of at
most $2n+1$~Dirac masses. Under additional hypotheses that are met in
our situation where $T=\T$, one can gain one dimension: there are $m\le 2n$~points
$z_k\in T$ and coefficients $b_k\in\C$ such that for every $f\in L$
one has $l(f)=\sum b_kf(z_k)$ and $\|l\|=\sum|b_k|$ (see
\cite[Exercice 6.8]{bo81}.) This implies in particular that there is a
function $f\in L$ whose maximum modulus points contain the $z_k$.

%On the other hand, given a basis $(f_j)$ of $L$ and any set of $n$~points $z_k$, the $n$ equations $l(f_j)=\sum b_kf_j(z_k)$ are expected to have a unique solution for the coefficients $b_k$, and this would yield an isomorphic extension of $l$. 

Let us now specialise to the case $L=\Cont_\Lambda(\T)$ with
$\Lambda$ a finite set. % Note that a function in $L$ has at most
% $\max\Lambda-\min\Lambda$ maximum modulus points; a trigonometric
% trinomial has at most 2 maximum modulus points up to periodicity.
Let us make the ad hoc hypothesis that the~$z_k$ are the $N$th roots
of unity, whose set forms the group $\U_N$: this obliges us to restrict our study to those
functionals~$l$ such that $l(\e_j)=l(\e_{j'})$ if $j\equiv
j'\bmod{N}$, where we write $\e_j(z)=z^j$ for $z\in\T$ and
  $j\in\Z$. Then the condition~$l(f)=\sum b_kf(z_k)$ reads
\begin{equation*}
l(\e_j)=\sum_{k=0}^{N-1}b_k\ei{2jk\pi/N}\text{ for $j\in\Lambda$,}
\end{equation*}
which may be interpreted as telling that the~$l(\e_j)$ are the Fourier
coefficients of the measure $\mu$ on $\U_N$ given by
\begin{equation*}
  \mu=\sum_{k=0}^{N-1}b_k\delta_{\ei{2k\pi/N}}
\end{equation*}
(where the Dirac measures act on $\U_N$). The set~$\Lambda$ might not be present in all classes modulo~$N$: let us set $l(\e_j)=0$ if $j$~is in a class in which~$\Lambda$ is absent. A “trivial” solution to these equations is then given by
\begin{equation*}
  b_k=\frac1N\sum_{j=0}^{N-1}\emi{2jk\pi/N}
  \begin{cases}
    l(\e_{j'})&\text{if there is $j'\equiv j\bmod{N}$ in $\Lambda$}\\
    0&\text{otherwise.}
  \end{cases}
\end{equation*}
The norm of $\mu$ is bounded by
\begin{equation*}
  \sum_{k=0}^{N-1}|b_k|
\end{equation*}
and is attained at $u\in\Cont(\U_N)$ if and only
if~$u(\ei{2k\pi/N})b_k=|b_k|$ for every~$k$, up to a nonzero complex
factor. This yields an upper bound for the norm of~$l$ that becomes an
equality if there is an~$f\in\Cont(\T)$ of norm~$1$ such that
$f(\ei{2k\pi/N})=u(\ei{2k\pi/N})$.

\section{My proof}

Here is a first application. The Sidon constant of a set~$\Lambda$ is also the supremum of the norm of the linear functionals~$l$ such that $l(e_j)$ is a unimodular complex number for all~$j\in\Lambda$:
\[|c_0|+|c_1|+\dots+|c_{n-1}|=\sup_{|l(\e_j)|=1}\biggl|\sum_{j=0}^{n-1}c_jl(\e_j)\biggr|.\]

\begin{prp}
  Let $\Lambda$ be a finite subset of $\Z$. The Sidon constant of
  $\Lambda$ is at most $(\card\Lambda-1)^{1/2}$.
\end{prp}

\begin{proof}
  One may suppose that $\min\Lambda=0$ and choose $N=\max\Lambda$.
  Let~$l$ be a linear functional with coefficients~$l(\e_j)$ of
  modulus 1: one may suppose that $l(\e_0)=l(\e_N)$. Then
  \begin{align}
    \|l\|
    &\le\frac1N\sum_{k=0}^{N-1}\biggl|\sum_{j\in\Lambda\setminus\{N\}}\emi{2jk\pi/N}l(\e_j)\biggr|\notag\\
    &\le\Biggl(\frac1N\sum_{k=0}^{N-1}\biggl|\sum_{j\in\Lambda\setminus\{N\}}\emi{2jk\pi/N}l(\e_j)\biggr|^2\Biggr)^{1/2}\label{ytt}\\
    &=\biggl(\sum_{j\in\Lambda\setminus\{N\}}|l(\e_j)|^2\biggr)^{1/2}=(\card\Lambda-1)^{1/2}.\notag\qedhere
  \end{align}
\end{proof}

% \begin{rem}
%   Herv\'e Queff\'elec showed me a more elementary proof of the same
%   fact, that follows the line of D. J.~Newman's computation in the introduction: if $f\in\Cont_\Lambda(\T)$, then there are only $\card\Lambda-1$ squares, and not $n$, in Inequality~\eqref{newm}! 
% \end{rem}
\begin{rem}
  If $\Lambda=\{0,1,\dots,n\}$, then Inequality~\eqref{ytt} is an
  equality if and only if $(l(\e_j))_{j=0}^{n-1}$ is a
  \emph{biunimodular} sequence, that is a unimodular function on
  $\U_n$ whose Fourier transform is also unimodular. In other words,
  the matrix $H=(l(\e_{j-k}))_{0\le j,k\le n-1}$ is a \emph{circulant
    complex Hadamard matrix}, where the indices $j-k$ are computed modulo
  $n$: it satisfies $H^*H=n\,\Id$. Such matrices always exist: see
  \cite{bs95}.
\end{rem}

\section{\texorpdfstring{The real unconditional constant of $\{0,1,2,3\}$}{The real unconditional constant of \{0,1,2,3\}}}

Here is a second application. Recall that the real unconditional
constant of a sequence of elements of a normed space is the maximal
distortion caused by multiplying the coefficients of a linear
combination of these elements by $\pm1$. By a slight abuse of
language, the real unconditional constant of a set~$\Lambda$ in the
space~$\Cont(\T)$ is thus the supremum of the norm of the linear
functionals~$l$ such that $l(\e_j)\in\{-1,1\}$ for all~$j\in\Lambda$.
\begin{prp}
\label{cir}
  Let $\Lambda=\{0,1,2,3\}$. The real unconditional constant of~$\Lambda$ in~$\Cont(\T)$ is $5/3$.
\end{prp}

\begin{proof}
  The polynomial ${-4}/{15} + {2z}/{5} + {z^2}/{5} +
{2z^3}/{15}$ studied in the next section will show that the real
  unconditional constant of $\Cont_\Lambda(\T)$ is at least $5/3$. As
  $l$ has the same norm as $\tilde l\colon f\mapsto l(f(\cdot+\pi))$,
  for which $\tilde l(\e_j)=(-1)^jl(\e_j)$, and as~$-l$, one may suppose that
  $l(\e_0)=l(\e_3)=1$. Let us now try to lift~$l$ to a sum of Dirac
  measures on the third roots of unity. Such a lifting is either the
  Dirac measure at $0$ or
  \begin{equation*}
    (l(\e_j)_{0\le j\le2})\in\{(1,-1,-1),(1,-1,1),(1,1,-1)\}
  \end{equation*}
  and these three cases yield the same norm
  \begin{equation*}
    \frac13(\abs{1-1-1}+\abs{1-\ei{2\pi/3}-\ei{4\pi/3}}+\abs{1-\ei{4\pi/3}-\ei{2\pi/3}})=5/3.\qedhere
  \end{equation*}
\end{proof}

\section{\texorpdfstring{The case $\{0,1,2,3\}$: a distinguished family of
  polynomials}{The case \{0,1,2,3\}: a distinguished family of
  polynomials}}\label{distpol}

\noindent 
% Let $\Lambda=\{0,1,2,3\}$.  Let $(\varrho,\vartheta)$ solve the extremal
% problem. If the point $t$ such that $\Phi^*(\varrho,\vartheta)=\Phi(t,\varrho,\vartheta)$
% were unique, then $\Phi^*_*=1$. If there were exactly two points $t,u$
% such that $\Phi^*(\varrho,\vartheta) = \Phi(t,\varrho,\vartheta) = \Phi(u,\varrho,\vartheta)$, then
% $\Phi^*_*=\sqrt2$. There are therefore exactly three points $t,u,t''$
% such that $\Phi^*(\varrho,\vartheta) = \Phi(t,\varrho,\vartheta) = \Phi(u,\varrho,\vartheta)=
% \Phi(t'',\varrho,\vartheta)$.
Let $f(z,\tau)$ be given by%\vskip\abovedisplayskip
%\centerline{$\displaystyle
\[\frac{\iu2\sqrt2\cos \tau-1-3\sin \tau}{15} + \frac{3+\sin
\tau}{10}z + \frac{3-\sin \tau}{10}z^2 +
\frac{\iu2\sqrt2\cos \tau-1+3\sin \tau}{15}z^3.\]%$}\vskip\abovedisplayskip
%\noindent 
One computes that the moduli sum of the coefficients is $1$,
independently of $\tau$. Note that $f(z,-\tau) = z^3f(z^{-1},\tau)$
and $f(z,\tau+\pi) = z^3\overline{f(z,\tau)}$, so that we shall
restrict the parameter $\tau$ to $[0,\pi/2]$. Let
$\Phi(t,\tau)=|f(\ei t,\tau)|^2$. We get
\begin{multline*}
\Phi(t,\tau) = \frac{2\sqrt2\sin{2\tau}}{75}
(\sin{t}-\sin{2t}+2\sin{3t}) + \frac{247-13\cos{2\tau}}{900}\\
+
(1+\cos{2\tau}) \Bigl( \frac{\cos{t}}{20} - \frac{\cos{2t}}{25}
\Bigl) + \frac{1+17\cos{2\tau}}{225} \cos{3t}.
\end{multline*}
Let us put 
$$M=\begin{pmatrix}
\displaystyle\frac{2\sin{2t}}{25}-\frac{\sin{t}}{20}-\frac{17\sin{3t}}{75}&
\displaystyle\frac{2\sqrt{2}}{75}(\cos{t}-2\cos{2t}+6\cos{3t})\\
\displaystyle\frac{2\sqrt{2}}{75}(\sin{t}-\sin{2t}+2\sin{3t})&
\displaystyle\frac{13}{900}-\frac{\cos{t}}{20}+\frac{\cos{2t}}{25}-\frac{17\cos{3t}}{225}
\end{pmatrix}.
$$
The critical points $(t,\tau)$ of $\Phi$ satisfy 
$$M\begin{pmatrix}\cos{2\tau}\\\sin{2\tau}\end{pmatrix}
=\begin{pmatrix}
  \displaystyle\frac{\sin{t}}{20}-\frac{2\sin{2t}}{25}+\frac{\sin{3t}}{75}\\
  0\vphantom{\frac00}
\end{pmatrix}.
$$
We have
$$\det{M} = \frac{1}{6750} \sin{t}
\Bigl(\cos{t}-\frac{1}{4}\Bigr) (4\cos{t}-11)
(16\cos^3{t}-72\cos^2{t}+33\cos{t}-41),$$
which vanishes exactly if
$\cos{t}\in\{-1,1/4,1\}$. Otherwise we get
\begin{equation}\label{CS}
  \left\{
    \begin{aligned}
      \cos{2\tau} & = \displaystyle- \frac{272\cos^3{t} - 72\cos^2{t} - 159\cos{t} + 23
      }{ 16\cos^3{t}-72\cos^2{t}+33\cos{t}-41}=C(t)\\
      \sin{2\tau} & = \displaystyle- \frac{24\sqrt{2} \sin{t} (4\cos{t}+1) (2\cos{t}-1) }{ 16\cos^3{t}-72\cos^2{t}+33\cos{t}-41}=S(t)
    \end{aligned}
  \right.
\end{equation}
Note that this solution is consistent, as $C^2+S^2=1$. For
such $\tau$, $\Phi(t,\tau)=9/25$. Checking the special cases
$\cos{t}\in\{-1,1/4,1\}$ yields that all local maxima are
given by the above formulas, that $\Phi$ attains its global
minimum, $0$, exactly for $\tau=0$ and $t=\pi$, and has exactly one other local
minimum, of value $49/225$, for $\tau=\pi/2$ and $t=0$. There is
exactly one other critical point, of value $5/18$, that
is a saddle point, given by $\tau=\arccos(17/37)/2$, $t=\arccos{1/4}$.

As $C(0)=1$, $C(\pm\pi/3)=-1$, $C(\pm\arccos(-1/4))=1$,
$C(\pi)=-1$, the inter\-mediate values theorem shows that
for a given $\tau$, there are exactly three solutions $t$ to
system \eqref{CS}, for which 
$\Phi(t,\tau)$ achieves then its global maximum, $9/25$.

Further details are given in~\cite{ne08b}.

\end{document}